\documentclass{article}
\usepackage{graphicx} 
\usepackage[utf8]{inputenc}
\usepackage{amsfonts,amsmath,amssymb,amsthm}
\usepackage{mathtools}
\usepackage{fancyhdr}
\usepackage[left=1in,right=1in,top=1in,bottom=1in]{geometry}
\usepackage{tabularx}
\usepackage{array}
\usepackage[hidelinks]{hyperref}

\theoremstyle{theorem}
\newtheorem{definition}{Definition}[section]
\newtheorem{theorem}[definition]{Theorem}
\newtheorem{proposition}[definition]{Proposition}
\newtheorem{lemma}[definition]{Lemma}

\newtheorem{corollary}[definition]{Corollary}

\theoremstyle{remark}
\newtheorem{remark}[definition]{Remark}
\newtheorem{example}[definition]{Example}
\newtheorem{question}[definition]{Question}

\title{Reflexive symmetric differentials and quotients of bounded symmetric domains}
\author{Aryaman Patel}
\date{}

\begin{document}

\maketitle

\begin{abstract}
    For each classical irreducible bounded symmetric domain $\mathcal{D}$, Klingler has computed the minimum number $m_{\mathcal{D}}$ such that any smooth projective quotient $X=\mathcal{D}/\Gamma$, for $\Gamma\in\textrm{Aut}^0(\mathcal{D})$, satisfies $H^0(X,\mathrm{Sym}^i\Omega^1_X)=0$ for $0<i<m_{\mathcal{D}}$. In this article, we extend Klingler's result to the case when $X$ is normal and projective. This, together with a normal version of Arapura's result about the relationship between the vanishing of global symmetric differentials on $X$ and the rigidity of finite dimensional representations of $\pi_1(X)$, gives rigidity statements for representations of $\pi_1(X)$ and $\pi_1(X_{reg})$ in a low dimensional range, when $X$ is a normal projective quotient of a bounded symmetric domain. 
\end{abstract}

\section*{Notation and conventions}

We work with algebraic varieties over the field $\mathbb{C}$ of complex numbers. Let $X$ be a normal projective variety and let $X_{reg}$ be the smooth locus of $X$. The tangent sheaf of $X$ is denoted by $\mathcal{T}_X$, and is a  reflexive sheaf. The sheaf of reflexive differentials is denoted by $\Omega^{[1]}_X$, where $\Omega^{[1]}_X=(\Omega^1_X)^{\vee\vee}=\mathcal{T}^\vee_X$. We are interested in reflexive symmetric powers of $\Omega^{[1]}_X$. We denote by $\mathrm{Sym}^{[i]}\Omega^1_X$ the $i$-th reflexive symmetric power of $\Omega^1_X$, i.e., $\mathrm{Sym}^{[i]}\Omega^1_X=(\mathrm{Sym}^i\Omega^1_X)^{\vee\vee}$, for $i\ge1$. 

\section{Introduction}

The aim of this work is to explore the relationship between representations of fundamental groups and vanishing of global reflexive symmetric differentials on $X$, when $X$ is normal and projective.

The articles \cite{klingler} and \cite{bkt}, on which the present work is largely based, arose from a question asked by H. Esnault. She asked whether a smooth complex projective variety with infinite fundamental group admits a non-zero global symmetric differential. The authors in \cite{bkt} proved a result based on this question, in the more general setting of compact K\"{a}hler manifolds. It is a simple observation that this result holds in the normal, projective setting. Namely, 

\begin{proposition}[{\cite[Theorem 0.1]{bkt}} in the compact K\"ahler case]\label{psing}
    Let $X$ be a normal projective variety. Suppose there is a finite dimensional representation of $\pi_1(X)$ over some field with infinite image. Then $X$ admits a non-zero reflexive symmetric differential.
\end{proposition}

\begin{remark}
    It is known that the natural map $j_*:X_{reg}\to X$ induced by the inclusion $j:X_{reg}\hookrightarrow X$, is surjective. If the kernel of $j_*$ is finite, we can replace $\pi_1(X)$ by $\pi_1(X_{reg})$ in the above statement.
\end{remark}

It is known from non-abelian Hodge theory developed by Simpson and its non-Archimedean analog that global symmetric differentials on $X$ determine the rigidity of finite dimensional linear representations of $\pi_1(X)$. This relationship can be extended to the case when $X$ is normal and projective.  
More precisely, we have the following.

\begin{theorem}[{\cite[Theorem 1.6]{klingler}} in the compact K\"ahler case]\label{thmazk}
    Let $X$ be a normal projective variety. Suppose that $H^0(X,\mathrm{Sym}^{[i]}\Omega^1_X)=0$ for all $1\le i\le r$, for some $r\in\mathbb{N}$. Then
    \begin{enumerate}
        \item Any representation $\rho:\pi_1(X)\to GL(r,\mathbb{C})$ is rigid (i.e., the affine variety \\$\mathrm{Hom}(\pi_1(X),GL(r,\mathbb{C}))//GL(r,\mathbb{C})$ is zero dimensional).
        \item Let $F$ be a non-Archimedean local field. Then any reductive representation $\rho:\pi_1(X)\to GL(r,F)$ has bounded image. Here $\rho$ is said to be reductive if the Zariski-closure of $\rho(\pi_1(X))$ is a reductive subgroup of $GL(r)/F$. It has bounded image if $\rho(\pi_1(X))$ is contained in a compact subgroup of $GL(r,F)$ for the topology of $GL(r,F)$ defined by the topology of the local field $F$. 
        
        If moreover $F$ has characteristic zero the reductiveness assumption is not needed.
    \end{enumerate}
\end{theorem}

Theorem \ref{thmazk}(1) was first proved by Arapura in the smooth projective setting (\cite[Proposition 2.4]{arapura}), and was later generalized to the compact K\"{a}hler setting by Klingler (\cite[Theorem 1.6(i)]{klingler}). Theorem \ref{thmazk}(2) for compact K\"ahler manifolds can be deduced from the work of Zuo in the characteristic zero case (\cite[Section 4.1.4]{zuo}), and was proved by Klingler in the general case (\cite[Theorem 1.6(ii)]{klingler}). 

The proofs of Proposition \ref{psing} and Theorem \ref{thmazk} are based on the following two facts. Let $X$ be a normal projective variety and let $f:\Tilde{X}\to X$ be any resolution of singularities. Then
\begin{enumerate}
\item The induced map $f_*:\pi_1(\Tilde{X})\to\pi_1(X)$ of topological fundamental groups is surjective. 
\item If $H^0(\Tilde{X},\textrm{Sym}^k\Omega^1_{\Tilde{X}})\neq0$, then $H^0(X,\textrm{Sym}^{[k]}\Omega^1_X)\neq0$ for any $k$.
\end{enumerate}
The first statement holds more generally for $X$ \emph{unibranch}, i.e., for any point $x\in X$, the analytic germ $(X,x)$ is irreducible (see \cite[Theorem 2.1]{arapura2}), and the second is an easy observation.\\
\\
The relationship between representations of the topological fundamental group and vanishing of global symmetric differentials lead Klingler to ask the following more precise question.

\begin{question}[{\cite[Question 1.10]{klingler}}]\label{q}
   Let $X$ be a smooth complex quasi-projective variety with $\Omega^1_X$ is positive in some sense. Can we detect the smallest $i\in\mathbb{N}$ for which $H^0(X,\mathrm{Sym}^i\Omega^1_X)\neq0$?
\end{question}

In the case that $X$ is a smooth quotient of a classical irreducible bounded symmetric domain $\mathcal{D}$ of rank $\ge2$, Question \ref{q} was answered by Klingler in \cite[Theorem 1.16]{klingler}. We can extend this to the case when $X$ is normal, and arrive at the following.

\begin{theorem}[{\cite[Theorem 1.16]{klingler}} in the smooth projective case]\label{thmk}
Let $\mathcal{D}$ be a classical irreducible bounded symmetric domain, $\Gamma\subset\mathrm{Aut}(\mathcal{D})$ a cocompact lattice acting fixed point freely in codimension $1$ on $\mathcal{D}$. Let $X=\mathcal{D}/\Gamma$ be the corresponding quotient variety. Then $H^0(X,\mathrm{Sym}^i\Omega^1_X)=0$ for $1\le i<m_{\mathcal{D}}$, where
\begin{itemize}
    \item $m_{\mathcal{D}}=\mathrm{inf}(p,q)$ if $\mathcal{D}=\mathcal{D}^I_{p,q}=SU(p,q)/S(U(p)\times U(q))$.
    \item $m_{\mathcal{D}}=[n/2]$ if $\mathcal{D}=\mathcal{D}^{II}_n=SO^*(2n)/U(n)$.
    \item $m_{\mathcal{D}}=[n/2]$ if $\mathcal{D}=\mathcal{D}^{III}_n=Sp(2n,\mathbb{R})/U(n)$.
    \item $m_{\mathcal{D}}=2$ if $\mathcal{D}=\mathcal{D}^{IV}_n=SO_0(2,n)/SO(2)\times SO(n)$.
\end{itemize}
\end{theorem}

The proof of this result in the smooth, projective setting is geometric, and is based on some deep vanishing results of Mok \cite{mok}, and classical plethysm formulas that can be found in in \cite{weyman}. 

The proof of Theorem \ref{thmk} is simply an application of Selberg's lemma (see \cite{alperin}) to reduce to the smooth projective case.

It is interesting that in the smooth projective case, these bounds are essentially sharp, i.e., given a domain $\mathcal{D}$, one can find a torsion-free lattice $\Gamma\subset\mathrm{Aut}^0(\mathcal{D})$ such that $X=\mathcal{D}/\Gamma$ satisfies $H^0(X,\mathrm{Sym}^{m_{\mathcal{D}}}\Omega^1_X)\neq0$. We do not know of a proof or examples of this. We do not know whether the bounds in the normal case are essentially sharp.

Note that Theorem \ref{thmk} is not useful in the rank 1 case. Indeed, there exist ball quotients $X$ for which $H^0(X,\Omega^1_X)\neq0$. However, for a \emph{Kottwitz lattice} $\Gamma\in SU(n,1)$ such that $n\ge2$ and $n+1$ is prime, \cite[Theorem 1.11]{klingler} shows that the corresponding compact ball quotient $X=\mathbb{B}^n/\Gamma$ admits no non-zero symmetric differentials up to order $n-1$. The proof of this result uses arithmetic methods.

Theorems \ref{thmazk} and \ref{thmk} together yield the following rigidity statements for representations of fundamental groups.

\begin{corollary}
    Let $X$ be a normal, projective quotient of an irreducible bounded symmetric domain $\mathcal{D}$. Then any representation $\rho:\pi_1(X_{reg})\to GL(r,\mathbb{C})$ is rigid for all $1\le r\le m_{\mathcal{D}}-1$. The same holds for $\pi_1(X)$.
\end{corollary}

The proof relies on the observation that non-rigidity of finite dimensional representations of a discrete group is preserved under passing to a finite index subgroup (Lemma \ref{lemg}).

We use this to extend \cite[Theorem 1.3(i)]{klingler} to a slightly larger class of lattices in $SU(n,1)$ and obtain the following.

\begin{proposition}\label{pkott}
    Let $\Gamma\subset SU(n,1)$ be a cocompact lattice acting fixed point freely in codimension one on $\mathbb{B}^n$, and suppose $\Gamma$ contains a Kottwitz lattice of finite index. \\
    If $n+1$ is prime, then any representation $\rho:\Gamma\to GL(n-1,\mathbb{C})$ is rigid.
\end{proposition}

The methods employed in the proof of \cite[Theorem 1.16]{klingler} can in principle be extended to quotients of reducible bounded symmetric domains. However, we observe that in the case of polydisk quotients, the results of Mok unfortunately cannot be used to say anything about the (non)existence of global reflexive symmetric differentials.

\section*{Acknowledgements}

I am grateful to Daniel Greb, Georg Hein, and 
Jochen Heinloth for fruitful discussions. I would also like to thank Manuel Hoff for discussions leading to the proof of Lemma \ref{lemg}. I am also grateful for Ya Deng's comments on the first version of the paper which improved it.

\section{Vanishing of reflexive symmetric differentials}

We first discuss notations and concepts involved in the proof of \cite[Theorem 1.16]{klingler}, adapted to the singular setting. The automorphism group of $\mathcal{D}$ is denoted by $G_0$, and we can write $\mathcal{D}=G_0/K_0$, where $K_0$ is a maximal compact subgroup of $G_0$, unique up to conjugation. There is an embedding $\mathcal{D}=G_0/K_0\subset G/Q$, where $G$ is a complexification of $G_0$, $Q$ is a parabolic subgroup of $G$, and $G/Q=:\mathcal{D}^\vee$ is the compact dual of $\mathcal{D}$. The complexification $K$ of $K_0$ is a Levi subgroup of $Q$. We denote by $\mathfrak{g}_0$, $\mathfrak{k}_0$, $\mathfrak{g}$, $\mathfrak{k}$, and $\mathfrak{q}$ the Lie algebras of $G_0$, $K_0$, $G$, $K$, and $Q$ respectively.\\
\\
The Lie algebra $\mathfrak{g}$ has a Hodge decomposition given by $\mathfrak{g}=\mathfrak{g}^{-1,1}\oplus\mathfrak{g}^{0,0}\oplus\mathfrak{g}^{1,-1}$, where $\mathfrak{g}^{0,0}=\mathfrak{k}$, and $\mathfrak{q}=\mathfrak{g}^{0,0}\oplus\mathfrak{g}^{-1,1}$. Let $\mathfrak{h}\subset\mathfrak{k}_0$ be a fixed Cartan subalgebra, and note that $\mathfrak{h}_{\mathbb{C}}:=\mathfrak{h}\otimes_{\mathbb{R}}\mathbb{C}$ is a Cartan subalgebra of $\mathfrak{k}$, as well as of $\mathfrak{g}$. Let $\mathfrak{h}^*_{\mathbb{R}}$ denote the space of all real valued linear forms on $\mathfrak{h}_{\mathbb{C}}$, and let $C$ be a fixed positive Weyl chamber in $\mathfrak{h}^*_{\mathbb{R}}$. Denote by $\langle\cdot,\cdot\rangle$ the inner product on $\mathfrak{h}^*_{\mathbb{R}}$, and let $\mu\in C$ be the highest root of $\mathfrak{g}$. Denote by $\Gamma\subset G_0$ a discrete group of automorphisms acting without fixed points in codimension one on $\mathcal{D}$, and let $X=\mathcal{D}/\Gamma$ be the corresponding quotient, which is a normal, quasi-projective variety. \\
\\
Let $V$ be a holomorphic vector bundle on the smooth locus $X_{reg}$ of $X$, equipped with a Hermitian metric $h$. Then the pair $(V,h)$ is called a \emph{Hermitian holomorphic vector bundle} on $X_{reg}$. For background and definitions of curvature, (semi-)negativity/positivity in the sense of Griffiths, and proper semi-negativity/positivity of Hermitian vector bundles, we refer the reader to \cite[Section 2]{klingler} and \cite[Chapter 10]{mok}.\\
\\
Let $\mathcal{D}$ be an irreducible bounded symmetric domain as usual, and let $\rho:K_0\to GL(V)$ be a complex finite dimensional representation of $K_0$, so equivalently $V$ is a $K$-module. Since $K$ is a Levi quotient of $Q$, $V$ can also be considered as a finite dimensional $Q$-module. We can associate to $\rho$ a $G$-equivariant vector bundle $G\times_{Q,\rho}V$ on the compact dual $\mathcal{D}^\vee=G/Q$. Its restriction $G\times_{K_0,\rho}V$ to $\mathcal{D}$ is a $G_0$-equivariant holomorphic vector bundle on $\mathcal{D}$. We can endow $G_0\times_{K_0,\rho}V$ with a $G$-equivariant Hermitan metric coming from a $K$-invariant Hermitian metric on $V$.  

\begin{definition}
    An \emph{automorphic bundle} $\mathcal{V}$ on the smooth locus $X_{reg}$ of $X=\mathcal{D}/\Gamma$ is a Hermitian holomorphic bundle on $X_{reg}$ of the form $\mathcal{V}=G_0/\Gamma\times_{K_0,\rho}V$ on, where $\rho:K_0\to GL(V)$ is a complex finite dimensional representation of $K_0$. We say that $\mathcal{V}$ is irreducible if the representation $\rho$ is irreducible.
\end{definition}

Properly semi-negative irreducible automorphic vector bundles have been characterized in terms of highest weight theory by Mok.  

\begin{proposition}[{\cite[Proposition 2, p.204]{mok}}]\label{prop0}
    Let $(\mathcal{V},s)$ be a homogeneous Hermitian vector bundle. Denote by $\mu$ the highest root of $\mathfrak{g}$, and by $\omega$ the lowest weight of $\rho:K_0\to GL(V)$. Then $(\mathcal{V},s)$ is properly seminegative if and only if $\langle\omega,\mu\rangle=0$. 
\end{proposition}

The following vanishing result by Mok is an important ingredient in the proof of Theorem \ref{thmk} in the smooth projective case.

\begin{proposition}[{\cite[Corollary 1, p.211]{mok}}]\label{mokp}
    Suppose $(\mathcal{V},s)$ is an irreducible non-trivial locally homogeneous vector bundle on $X=\mathcal{D}/\Gamma$. If $(\mathcal{V},s)$ is of properly seminegative curvature in the sense of Griffiths, then $H^0(X,\mathcal{V}^\vee)=0$.
\end{proposition}

Note that $(\mathcal{V},h)$ being semi-negative is equivalent to the dual $(\mathcal{V}^\vee,h^\vee)$ being semi-positive. The tangent bundle of a smooth, quasi-projective quotient $X$ of $\mathcal{D}$ is shown to be properly semi-negative when $\textrm{rank}(\mathcal{D})\ge2$ in \cite[p. 204]{mok}. The following result due to Klingler then completes the proof of Theorem \ref{thmk} in the smooth, projective setting.

\begin{theorem}[{\cite[Theorem 2.7]{klingler}}]\label{thmbk}
    Let $\mathcal{D}$ be a classical irreducible bounded symmetric domain, $\Gamma\subset\mathrm{Aut}^0(\mathcal{D})$ a torsion-free lattice, and $X=\mathcal{D}/\Gamma$ the corresponding quotient. Then the automorphic bundle $\mathrm{Sym}^i\mathcal{T}_X$ is a direct sum of properly seminegative irreducible automorphic bundles, for $0<i< m_{\mathcal{D}}$.
\end{theorem}

The bundle $\textrm{Sym}^i\mathcal{T}_X$ is the automorphic bundle associated to the $K$-module $\textrm{Sym}^i\mathfrak{g}^{-1,1}$. The proof of Theorem \ref{thmbk} consists of computing the lowest weight of all irreducible $K$-modules which are direct summands of $\textrm{Sym}^i\mathfrak{g}^{-1,1}$, and then applying Proposition \ref{prop0}. Klingler's result \cite[Theorem 1.16]{klingler} then follows by applying Proposition \ref{mokp} to $\mathcal{V}=\mathrm{Sym}^i\mathcal{T}_X$.\\
\\
Using Selberg's lemma and the fact that the automorphism group of each classical irreducible bounded symmetric domain has finitely many connected components, we can extend \cite[Theorem 1.16]{klingler} to the normal setting. More precisely, we have

\begin{theorem}[=Theorem \ref{thmk}]\label{thmk2}
Let $\mathcal{D}$ be a classical irreducible bounded symmetric domain. Let $\Gamma\subset\mathrm{Aut}(\mathcal{D})$ be a cocompact lattice whose action on $\mathcal{D}$ is fixed point free in codimension 1, and let $X=\mathcal{D}/\Gamma$ be the corresponding quotient. Then $H^0(X,\mathrm{Sym}^{[i]}\Omega^1_X)=0$ for $1\le i<m_{\mathcal{D}}$, where
\begin{itemize}
    \item $m_{\mathcal{D}}=\textrm{inf}(p,q)$ if $\mathcal{D}=\mathcal{D}^I_{p,q}=SU(p,q)/S(U(p)\times U(q))$.
    \item $m_{\mathcal{D}}=[n/2]$ if $\mathcal{D}=\mathcal{D}^{II}_n=SO^*(2n)/U(n)$.
    \item $m_{\mathcal{D}}=[n/2]$ if $\mathcal{D}=\mathcal{D}^{III}_n=Sp(2n,\mathbb{R})/U(n)$.
    \item $m_{\mathcal{D}}=2$ if $\mathcal{D}=\mathcal{D}^{IV}_n=SO_0(2,n)/SO(2)\times SO(n)$.
\end{itemize}
\end{theorem}

\begin{proof}
The automorphism groups of $\mathcal{D}^I_{p,p}$ and $\mathcal{D}^{IV}_n$ have two connected components, while those of $\mathcal{D}^I_{p,q}$ ($p\neq q$), $\mathcal{D}^{II}_n$ and $\mathcal{D}^{III}_n$ are connected. We treat the cases when $\textrm{Aut}(\mathcal{D})$ is connected separately from when it is disconnected.

Suppose $X$ is a quotient of $\mathcal{D}^I_{p,p}$. Then up to a $2:1$ quasi-etale cover, we can write $\mathcal{T}_{X_{reg}}\cong P\times_K\mathfrak{g}^{-1,1}$, where $K=S(GL(p,\mathbb{C})\times GL(p,\mathbb{C}))$. By Selberg's lemma (\cite{alperin}), the group $\Gamma$ admits a normal, torsion-free subgroup $\Gamma'$ of finite index. Thus the quotient map $\mathcal{D}^I_{p,p}\to X$ factors as $\mathcal{D}^I_{p,p}\to Y\to X$, where $Y=\mathcal{D}^I_{p,p}/\Gamma'$ is smooth and projective, and the map $\gamma:Y\to X$ is finite, quasi-\'etale, and Galois. Note that $\gamma^{[*]}(\textrm{Sym}^{[i]}\Omega^1_X)\cong\textrm{Sym}^i\Omega^1_Y$. From \cite[Theorem 1.16]{klingler}, we know that $H^0(Y,\textrm{Sym}^i\Omega^1_Y)=0$ for $i<p$, thus it follows that $H^0(X,\textrm{Sym}^{[i]}\Omega^1_X)=0$ for $i<p$. 

If $X$ is a quotient of $\mathcal{D}^I_{p,q}$ for $p\neq q$, we don't need to pass to a $2:1$ quasi-etale cover. Using Selberg's lemma as earlier, we conclude that $H^0(X,\textrm{Sym}^{[i]}\Omega^1_X)=0$ for $i<\textrm{inf}(p,q)$.

Next, suppose $X$ is a quotient of $\mathcal{D}^{IV}_n$. Although $G_0=\textrm{Aut}(\mathcal{D}^{IV}_n)$ has two connected components, the corresponding $K$ is $SO(n,\mathbb{C})\times SO(2,\mathbb{C})$, which is connected. Hence we don't need to pass to a $2:1$ quasi-etale cover. We can write $\mathcal{T}_{X_{reg}}\cong\mathcal{H}om(\mathcal{W},\mathcal{L})$, where $\mathcal{W}$ is an orthogonal vector bundle of rank $n$, and $\mathcal{L}$ is a line bundle on $X_{reg}$. Since $\mathcal{W}$ is orthogonal, it is in particular self-dual, thus satisfies $\textrm{det}(\mathcal{W})\cong\mathcal{O}_{X_{reg}}$. It follows that $K_X\cong\mathcal{L^\vee}^{\otimes n}$, and since $K_X$ is ample, so is $\mathcal{L}$. From the proof of Theorem \ref{thmbk}, we have that $\mathcal{L}^2$, which is also ample, is an irreducible direct factor  of $\textrm{Sym}^2\Omega^1_{X_{reg}}$, and has global sections if $\Gamma$ is suitably chosen. 
On the other hand, we know that for compact quotients $Y$ of irreducible bounded symmetric domains of rank $\ge2$, $\mathcal{T}_Y$ is properly seminegative, i.e., $H^0(Y,\Omega^1_Y)=0$. Consequently, we have $H^0(X,\textrm{Sym}^{[i]}\Omega^1_X)=0$ for $i<2$.

Finally, suppose that $X$ is a quotient of either $\mathcal{D}^{II}_n$ or $\mathcal{D}^{III}_n$. Note that the automorphism groups of both domains are connected. We proceed similarly as in the $\mathcal{D}^I_{p,q}$ by using Selberg's lemma, and arrive at the same result as Theorem \ref{thmk}, i.e., $H^0(X,\textrm{Sym}^{[i]}\Omega^1_X)=0$ for $0<i\le[\frac{n}{2}]$. 
\end{proof}

\begin{remark}
     (i) The uniformization result \cite[Theorem 1.1]{thesis} allows us to consider quotients of bounded symmetric domains by subgroups of the full automorphism group, not just the connected component. \\
     (ii) The essential sharpness of the bounds in \cite[Theorem 1.16]{klingler} is claimed but not justified. It is not clear whether the global sections of $\mathrm{Sym}^{m_{\mathcal{D}}}\Omega^1_Y$ descend to non-trivial global sections of $\mathrm{Sym}^{[m_{\mathcal{D}}]}\Omega^1_X$. Thus we do not know whether the bounds of Theorem \ref{thmk2} are essentially sharp. 
\end{remark}

In contrast to the case of irreducible bounded symmetric domains, symmetric powers of the tangent bundle of the smooth locus of a polydisk quotient are not direct sums of properly seminegative irreducible non-trivial automorphic bundles. More precisely, we have the following.

\begin{proposition}\label{pdq}
    Let $X=\mathbb{H}^n/\Gamma$ be a quotient of the polydisk by a discrete group of automorphisms $\Gamma\subset\mathrm{Aut}(\mathbb{H}^n)=PSL(2,\mathbb{R})^n\rtimes S_n$, where $S_n$ denotes the symmetric group in $n$-letters, such that $\Gamma$ acts fixed point freely in codimension one on $\mathbb{H}^n$. Then $\mathrm{Sym}^i\mathcal{T}_{X_{reg}}$ is a direct sum of non-trivial automorphic vector bundles of strictly negative curvature.
\end{proposition}
\begin{proof}

Any such $X$ admits a Galois, quasi-\'etale cover $Y$ such that $\mathcal{T}_{Y_{reg}}$ splits as a direct sum of line bundles of negative degree, i.e., $Y\cong\mathbb{H}^n/\Gamma'$, where $\Gamma'\subset PSL(2,\mathbb{R})^n$. Thus any indecomposable direct summand $V$ of $\mathrm{Sym}^i\mathcal{T}_{X_{reg}}$ decomposes as a direct sum of line bundles on $Y_{reg}$. If the claim holds for $Y$, then it holds for $X$ because pullback preserves (semi)negativity/positivity of curvature. Hence we may assume that $\Gamma\subset PSL(2,\mathbb{R})^n$.

Let $G_0=SL(2,\mathbb{R})^n$, with maximal compact subgroup $K_0=U(1)^n$. Their complexifications are $G=SL(2,\mathbb{C})^n$ and $K=(\mathbb{C}^*)^n$ respectively. The Lie algebra $\mathfrak{g}=\mathfrak{sl}(2,\mathbb{C})^n$ is semisimple, and it splits as a direct sum of simple ideals $\mathfrak{g}=\bigoplus_{i=1}^n\mathfrak{g}_i$, where $\mathfrak{g}_i\cong\mathfrak{sl}(2,\mathbb{C})$ for all $i$. Each $\mathfrak{g}_i$ admits a Hodge decomposition given by $\mathfrak{g}_i=\mathfrak{g}_i^{-1,1}\oplus\mathfrak{g}^{0,0}\oplus\mathfrak{g}_i^{1,-1}$, where each summand is isomorphic, as a $K$-module, to $\mathbb{C}$.

The tangent bundle of the smooth locus $X_{reg}$ splits as a direct sum $\mathcal{T}_{X_{reg}}=\bigoplus_{i=i}^n\mathcal{L}_i$, where each $\mathcal{L}_i$ is a line bundle of negative degree. Moreover, each $\mathcal{L}_i$ is the irreducible automorphic line bundle associated to the irreducible $K$-module $\mathfrak{g}_i^{-1,1}\cong\mathbb{C}$. The bundle $\textrm{Sym}^i\mathcal{T}_{X_{reg}}$ is thus the automorphic bundle associated to the $K$-module $\textrm{Sym}^i\mathfrak{g}^{-1,1}=\textrm{Sym}^i\mathbb{C}^{\oplus n}$, and also decomposes as a direct sum of line bundles.

We know from \cite[Proposition 1, p.202]{mok} that $\langle\tau,\mu\rangle$ is fiberwise the maximum value attained by the curvature tensor of any irreducible direct summand $\mathcal{L}$ of $\mathrm{Sym}^i\mathcal{T}_{X_{reg}}$, where $\tau$ denotes the lowest weight of the irreducible $K$-module to which $\mathcal{L}$ is associated, and $\mu$ is the highest root of $\mathfrak{g}$. Following the method of the proof of \cite[Theorem 2.7]{klingler}, we want to show that $\tau$ is non-trivial, and satisfies $\langle\tau,\mu\rangle<0$ for all $i\ge1$.

The Lie algebra $\mathfrak{k}_0$ of $K_0=U(1)^n$ is its own Cartan subalgebra. It consists of $n$-tuples of diagonal $2\times2$ matrices of trace 0. Let $\{E_{i_kj_k}\}_{k=1}^n$ be an $n$-tuple of $2\times2$ matrices, where $E_{i_kj_k}$ denotes the $2\times2$ matrix in the $k$-th position with entries 0, except 1 at the $(i,j)$-th entry, $1\le i,j\le2$. Let $\{L_{i_k}\}_{k=1}^n$ be an $n$-tuple of linear forms on $\mathfrak{k}$, where $L_{i_k}$ takes value 1 at $E_{i_ki_k}$ and 0 at $E_{j_kj_k}$, $j\neq i$. We choose the positive Weyl chamber $C\subset\mathfrak{k}^*_{\mathbb{R}}$ to be the set of $\sum_{k=1}^n\sum_{i=1}^2a_{i_k}L_{i_k}$ with $a_{1_k}\ge a_{2_k}$. The scalar product on $\mathfrak{k}^*_{\mathbb{R}}$ is then given by $\langle\sum_k\sum_ia_{i_k},\sum_k\sum_jb_{j_k}L_{j_k}\rangle=\sum_k\sum_ia_{i_k}b_{i_k}$. \\
\\
Each irreducible $K$-factor of $\mathfrak{g}^{-1,1}$ is a copy of $\mathbb{C}$ corresponding to one of the $\mathcal{L}_i$, and its lowest weight $\tau$ is given by:
\begin{align*}
    \tau=-L_{1_i}+L_{2_i}
\end{align*}
for some $1\le l\le n$. The highest root $\mu$ of $\mathfrak{g}=\mathfrak{sl}(2,\mathbb{C})^n$ is given by:
\begin{align*}
    \mu=L_{1_1}-L_{2_1}+L_{1_2}-L_{2_2}+\dots+L_{1_n}-L_{2_n}.
\end{align*}
Thus we have 
\begin{align*}
    \langle\tau,\mu\rangle=-2,
\end{align*}
which is always negative. Irreducible $K$-factors of $\mathrm{Sym}^i\mathfrak{g}^{-1,1}$ are again copies of $\mathbb{C}$ corresponding to tensor products of powers of the $\mathcal{L}_i$'s, so it follows that $\langle\tau,\mu\rangle<0$ for any such factor. Thus $\textrm{Sym}^i\mathcal{T}_{X_{reg}}$ is a direct sum  of irreducible automorphic line bundles of strictly negative curvature. 
\end{proof}

\begin{remark}
    Equivalently, for $X$ as in Proposition \ref{pdq}, $\mathrm{Sym}^i\Omega^1_{X_{reg}}$ is a direct sum of non-trivial automorphic bundles of strictly positive curvature. Thus we cannot use Mok's results (Proposition \ref{mokp} or \cite[Corollary 1', p.212]{mok}) to say anything about the (non)vanishing of $H^0(X,\mathrm{Sym}^{[i]}\Omega^1_X)$.
\end{remark}

There exist quotients $X$ of the polydisk satisfying $H^0(X,\Omega^1_X)\neq0$, as shown in the following example. Thus we cannot use Theorem \ref{thmazk} to say something about the rigidity of fundamental groups of quotients of the polydisk $\mathbb{H}^n$, even in the smooth case.

\begin{example}[Products of curves]\label{pc}
Let $C_1$,...,$C_n$ be smooth projective curves, each of genus $g_i\ge2$, and consider the product $X=C_1\times\dots\times C_n$. By Simpson's classical uniformization result we know that the universal cover of $X$ is the polydisk $\mathbb{H}^n$. Let $p_i$ denote the projection map $p_i:X\to C_i$. The cotangent bundle of $X$ is given by $\Omega^1_X=\bigoplus_{i=1}^np_i^*\Omega^1_{C_i}$. From the classical Riemann-Roch theorem we know that $\textrm{dim}(H^0(C_i,\Omega^1_{C_i}))=g_i$ for each $i$, so in particular
\begin{align*}
    H^0(X,\Omega^1_X)=\bigoplus_{i=1}^nH^0(X,p_i^*\Omega^1_{C_i})=\sum_{i=1}^ng_i>0.
\end{align*}
Therefore $H^0(X,\mathrm{Sym}^k\Omega^1_X)>0$ for all $k\ge1$.
\end{example}
Taking this a step further, we can consider quotients of products of curves by finite groups. Such varieties and their fundamental groups have been studied extensively, for example in \cite{cat}. 
\begin{example}[Product quotients]
    Let $X=C_1\times\dots\times C_n$ be as in Example \ref{pc} and let $G$ be a finite group of automorphisms acting freely and diagonally on $X$, i.e., $G$ acts on each factor $C_i$ by automorphisms. Let $\pi:X\to X/G=:Y$ be the corresponding quotient. Then $Y$ is smooth, uniformized by $\mathbb{H}^n$, and the cotangent bundle of $Y$ splits as $\Omega^1_Y\cong\bigoplus_{i=1}^n\mathcal{L}_i$, where $\pi^*\mathcal{L}_i\cong p_i^*\Omega^1_{C_i}$. 
    Global sections of $\Omega^1_Y$ come from global sections of the summands $\mathcal{L}_i$. We have 
    \begin{align*}
        H^0(Y,\mathcal{L}_i)=H^0(X,\pi^*\mathcal{L}_i)^G=H^0(Y,p_i^*\Omega^1_{C_i})^G.
    \end{align*}
    Since the map $p_i:X\to C_i$ is projective and has connected fibers, we have $p_{i*}\mathcal{O}_X\cong\mathcal{O}_{C_i}$. From the projection formula it follows that $p_{i*}p_i^*\Omega^1_{C_i}\cong\Omega^1_{C_i}\otimes p_{i*}\mathcal{O}_X\cong\Omega^1_{C_i}$. In particular, $H^0(Y,\mathcal{L}_i)\cong H^0(C_i,\Omega^1_{C_i})^G$.  
    
    If the action of $G$ on $C_i$ is free, then Riemann-Roch says that $H^0(C_i,\Omega^1_{C_i})^G=g(C_i/G)$. So if $C$ is a smooth projective curve of genus $g(C)\ge2$ such that $G$ acts freely on $C$ and $g(C/G)=0$, then the product quotient $X=(C\times\dots\times C)/G$ admits no non-zero symmetric differentials i.e., $H^0(X,\Omega^1_X)=0$. Consequently, all one-dimensional representations of $\pi_1(X)$ are rigid.
\end{example}



\section{Representations of the fundamental group}

The goal of this section is to extend the relationship between global symmetric differential forms and representations of the topological fundamental group to the normal setting. We begin with the following more general version of Arpaura's classical result. 

\begin{proposition}[=Theorem \ref{thmazk}(1)]\label{propa}
    Let $X$ be a normal projective variety. If $\pi_1(X)$ has infinitely many non-isomorphic semisimple representations into $GL(n,\mathbb{C})$, then $H^0(X,\mathrm{Sym}^{[i]}\Omega^1_X)\neq0$ for some $1\le i\le n$.
\end{proposition}
\begin{proof}
    Let $f:\Tilde{X}\to X$ be a resolution of singularities of $X$. Then we know, for example by \cite[Theorem 2.1]{arapura2}, that the induced map $f_*:\pi_1(\Tilde{X})\to\pi_1(X)$ is surjective. In particular, $\pi_1(\Tilde{X})$ has infinitely many non-isomorphic semisimple representations into $GL(n,\mathbb{C})$. By \cite[Proposition 2.4]{arapura}, it follows that $H^0(\Tilde{X},\textrm{Sym}^i\Omega^1_{\Tilde{X}})\neq0$ for some $1\le i\le n$.
    
    Let $\Tilde{X}^o\subset\Tilde{X}$ be the open subset restricted to which $f$ is an isomorphism. Then by restricting global sections to $\Tilde{X}^o$, we have that $H^0(\Tilde{X}^o,\textrm{Sym}^i\Omega^1_{\Tilde{X}^o})\neq0$ for some $i$. Via the isomorphism $\Tilde{X}^o\cong X_{reg}$ we have
    \begin{align*}
        H^0(\Tilde{X}^o,\textrm{Sym}^i\Omega^1_{\Tilde{X}^o})=H^0(X_{reg},\textrm{Sym}^i\Omega^1_{X_{reg}})=H^0(X,\textrm{Sym}^{[i]}\Omega^1_X)\neq0
    \end{align*}
    for some $i$. This proves the assertion.
\end{proof}

\begin{remark}
    The above statement can be rephrased as follows: if $H^0(X,\mathrm{Sym}^{[i]}\Omega^1_X)=0$ for $1\le i\le r$, then every representation $\rho:\pi_1(X)\to GL(r,\mathbb{C})$ is rigid. 
\end{remark}

Using the same argument as in the proof of Proposition 3.1, and the non-Archimedean version of Arapura's result due to Klingler and Zuo, we arrive at the following non-Archimedean version of Proposition \ref{propa}.

\begin{proposition}[=Theorem \ref{thmazk}(2)]\label{propa2}
    Let $X$ be a normal projective variety, and suppose that $H^0(X,\mathrm{Sym}^{[i]}\Omega^1_X)=0$ for all $0<i\le r$, for some $r\in\mathbb{N}$. Let $F$ be a non-Archimedean local field. Then any reductive representation $\rho:\pi_1(X)\to GL(r,F)$ has bounded image. 
    
    If moreover $F$ has characteristic zero the reductiveness assumption is not needed.
\end{proposition}
\begin{proof}
    Let $f:\Tilde{X}\to X$ be a resolution of singularities of $X$. Then $H^0(\Tilde{X},\textrm{Sym}^i\Omega^1_{\Tilde{X}})=0$ for all $0<i\le r$. Indeed if not, then a non-zero section of $\textrm{Sym}^i\Omega^1_{\Tilde{X}}$ restricted to the open subset of $\Tilde{X}$ where $\pi$ is an isomorphism would give a non-zero section of $\textrm{Sym}^i\Omega^1_{X_{reg}}$, and hence of $\textrm{Sym}^{[i]}\Omega^1_X$, which contradicts the assumption.\\
    Then it follows from \cite[Theorem 1.6(ii)]{klingler}, that any reductive representation $\rho:\pi_1(\Tilde{X})\to GL(r,F)$ has bounded image. Since $f_*:\pi(\Tilde{X})\to\pi_1(X)$ is surjective, the same holds for any reductive representation $\pi_1(X)\to GL(r,F)$.

    The second part of the Proposition was proved by Klingler in the compact K\"ahler case, and by the surjectivity of $f_*$ it holds also in the normal case.
\end{proof}

We have the following arithmetic corollary, which is a singular version of \cite[Corollary 1.8]{klingler}.

\begin{corollary}\label{coro}
    Let $X$ be a normal projective variety, and suppose $H^0(X,\mathrm{Sym}^{[i]}\Omega^1_X)=0$ for $0<i\le r$. Then any representation $\rho:\pi_1(X)\to GL(r,\mathbb{C})$ is conjugate to a representation $\rho_0:\pi_1(X)\to GL(r,\mathcal{O}_K)$, where $\mathcal{O}_K$ is the ring of integers of some number field $K\subset\mathbb{C}$.
\end{corollary}
\begin{proof}
    Any resolution of singularities $\Tilde{X}$ of $X$ satisfies the assumptions. The conclusion thus holds for $\Tilde{X}$ by \cite[Corollary 1.8]{klingler}. We use again that the induced map $\pi_1(\Tilde{X})\to\pi_1(X)$ is surjective.
\end{proof}

Theorem \ref{thmk2} and Proposition \ref{propa} together yield the following observation about representations of the fundamental group of a projective, klt quotient of an irreducible bounded symmetric domain $\mathcal{D}$.

\begin{corollary}
Let $X$ be a normal, projective quotient of an irreducible bounded symmetric domain $\mathcal{D}$. Then any representation $\rho:\pi_1(X)\to GL(r,\mathbb{C})$ is rigid for all $1\le r\le m_{\mathcal{D}}-1$.  
\end{corollary}

We can use the fact that every normal, projective quotient of a bounded symmetric domain admits a quasi-\'etale cover which is smooth and projective, to derive a link between global reflexive symmetric differentials and rigidity of the topological fundamental group of the smooth locus. We make use of the following auxiliary statements.

\begin{lemma}[{\cite[Lemma 2.7]{hnb}}]\label{aux1}
    Let $G$ be a locally profinite group, and let $H$ an open subgroup of $G$ of finite index. If $(\rho,V)$ is a smooth representation of $G$, then $V$ is $G$-semisimple if and only if it is $H$-semisimple.
\end{lemma}

As a corollary, we obtain the following useful observation.

\begin{lemma}\label{lemg}
Let $G$ be a locally profinite group and let $H$ be an open subgroup of $G$ of finite index. If $G$ has infinitely many non-isomorphic semisimple representations into $GL(n,\mathbb{C})$, then so does $H$.
\end{lemma}
\begin{proof}
    By Lemma \ref{aux1}, the restriction of any semisimple representation of $G$ into $GL(n,\mathbb{C})$ is a semisimple representation of $H$ into $GL(n,\mathbb{C})$. Thus it remains to show that only finitely many non-isomorphic semisimple representations of $G$ restrict to the same representation of $H$ up to isomorphism.
    
    To see this, let $\sigma:H\to GL(n,\mathbb{C})$ be a semisimple representation, and consider the induced representation $\textrm{Ind}^G_H\sigma$ of $G$. The restriction of $\textrm{Ind}^G_H\sigma$ to $H$ is semisimple because the associated module is a direct sum of $H$-modules $\mathbb{C}^n\otimes h$, where $h$ runs over the set of coset representatives of $G/H$. Then again by Lemma \ref{aux1} it follows that $\textrm{Ind}^G_H\sigma$ semisimple. Every semisimple representation $\rho$ of $G$ into $GL(n,\mathbb{C})$ that restricts to $\sigma$ is a quotient of $\textrm{Ind}^G_H\sigma$. Since $\textrm{Ind}^G_H\sigma$ is semisimple, there can only be finitely many such representations $\rho$. This concludes the proof.
\end{proof}

We can now extend the vanishing result \cite[Theorem 1.11]{klingler} to a wider class of lattices in $SU(n,1)$. We have the following.

\begin{proposition}\label{ballprop}
    Let $\Gamma\subset SU(n,1)$ be a cocompact lattice acting fixed point freely in codimension one on $\mathbb{B}^n$, and suppose $\Gamma$ contains a torsion-free Kottwitz lattice $\Gamma'$ as a finite index subgroup. Let $X=\mathbb{B}^n/\Gamma$ be the corresponding ball quotient.\\
    If $n+1$ is prime, then $H^0(X,\mathrm{Sym}^{[i]}\Omega^1_X)=0$ for $1\le i\le n-1$.
\end{proposition}
\begin{proof}
    The quotient map $\mathbb{B}^n\to X$ factors as 
    \begin{align*}
        \mathbb{B}^n\to\mathbb{B}^n/\Gamma'=:Y\to X,
    \end{align*}
    where $Y$ is smooth and projective, and the map $f:Y\to X$ is quasi-\'etale. It follows that $f^{[*]}\mathrm{Sym}^{[i]}\Omega^1_X\cong\mathrm{Sym}^i\Omega^1_Y$. Since $H^0(Y,\mathrm{Sym}^i\Omega^1_Y)=0$ for $1\le i\le n-1$ by \cite[Theorem 1.11]{klingler}, the same holds for $X$. 
\end{proof}

We can thus apply Propositions \ref{propa}, \ref{propa2}, and Corollary \ref{coro} to say something about representations of $\pi_1(X)$, when $X$ is a ball quotient as in Proposition \ref{ballprop}. In particular, we have that any representation $\rho:\pi_1(X)\to GL(r,\mathbb{C})$ is rigid for $1\le r\le n-1$.\\
\\
When $X$ is singular, the fundamental group of the smooth locus $X_{reg}$ is usually the more interesting one to consider. In case of ball quotients, we have

\begin{proposition}
    Let $\Gamma\subset SU(n,1)$ be a cocompact lattice in $SU(n,1)$ and suppose $\Gamma$ contains a torsion-free Kottwitz lattice $\Gamma'$ of finite index. Let $X=\mathbb{B}^n/\Gamma$ be the corresponding ball quotient.\\
    If $n+1$ is prime, then every representation $\rho:\pi_1(X_{reg})\to GL(n-1,\mathbb{C})$ is rigid.
\end{proposition}
\begin{proof}
    Let $Y:=\mathbb{B}^n/\Gamma'$ and note that every representation $\Gamma'=\pi_1(Y)\to GL(n-1,\mathbb{C})$ is rigid by Propositions \ref{ballprop} and \ref{propa}. Since $\Gamma=\pi_1(X_{reg})$, the conclusion follows from Lemma \ref{lemg}.
\end{proof}

As a consequence, we obtain the following rigidity result, which is an extension of \cite[Theorem 1.3(i)]{klingler} to a wider class of lattices in $SU(n,1)$.

\begin{corollary}[=Proposition \ref{pkott}]
    Let $\Gamma\subset SU(n,1)$ be a cocompact lattice acting fixed point freely in codimension one on $\mathbb{B}^n$, and suppose $\Gamma$ contains a Kottwitz lattice of finite index. \\
    If $n+1$ is prime, then any representation $\rho:\Gamma\to GL(n-1,\mathbb{C})$ is rigid.
\end{corollary}

Another consequence of Proposition \ref{propa} and Theorem \ref{thmk2} is the following observation about the rigidity of representations of $\pi_1(X_{reg})$ when $X$ is a normal projective quotient of an irreducible bounded symmetric domain. 

\begin{corollary}
    Let $X$ be a normal projective quotient of an irreducible bounded symmetric domain $\mathcal{D}$. Then any representation $\pi_1(X_{reg})\to GL(r,\mathbb{C})$ is rigid for all $1\le r\le m_{\mathcal{D}}-1$.  
\end{corollary}
\begin{proof}
    By Selberg's lemma, there is a Galois, quasi-etale cover $\gamma:Y\to X$ such that $Y$ is a smooth, projective quotient of $\mathcal{D}$. If we write $X=\mathcal{D}/\Gamma$ for some $\Gamma\subset\textrm{Aut}(\mathcal{D})$, then $\pi_1(X_{reg})\cong\Gamma$, and $\pi_1(Y)$ is a normal subgroup of $\Gamma$ of finite index. 
    
    By Theorem \ref{thmk} we know that $H^0(Y,\textrm{Sym}^k\Omega^1_Y)=0$ for all $1\le k<m_{\mathcal{D}}$. Then by Proposition \ref{propa} it follows that every representation $\rho:\pi_1(Y)\to GL(r,\mathbb{C})$ is rigid for all $r\le m_{\mathcal{D}}-1$. From Lemma \ref{lemg}, it follows that the same holds for $\pi_1(X_{reg})$, and we are done.
\end{proof}

Extending the result \cite[Theorem 6.1]{bkt}, we arrive at the following link between representations of the fundamental group and the existence of global reflexive symmetric differentials in the normal setting. The proof is the same as that of Proposition \ref{propa}.

\begin{proposition}
    Let $X$ be a normal projective variety. Suppose there is a representation $\rho:\pi_1(X)\to GL(n,\mathbb{C})$ with infinite image, for some $n$. Then $X$ has a non-zero reflexive symmetric differential.
\end{proposition}
\begin{proof}
    Let $f:\Tilde{X}\to X$ be a resolution of singularities. Then $f_*:\pi_1(\Tilde{X})\to\pi_1(X)$ is surjective, and the representation $\rho\circ f_*:\pi_1(\Tilde{X})\to GL(n,\mathbb{C})$ has infinite image. By \cite[Theorem 6.1]{bkt}, the assertion holds for $\Tilde{X}$, i.e, $H^0(\Tilde{X},\textrm{Sym}^i\Omega^1_{\Tilde{X}})\neq0$ for some $i>0$. Then by the same argument as in the proof of Proposition \ref{propa}, it follows that $H^0(X,\textrm{Sym}^{[i]}\Omega^1_X)\neq0$, and we are done.
\end{proof}

In the same spirit, we have the following normal version of \cite[Theorem 0.1]{bkt}.

\begin{proposition}[=Proposition \ref{psing}]\label{pbktsing}
    Let $X$ be a normal projective variety. Suppose there is a finite dimensional representation of $\pi_1(X)$ over some field with infinite image. Then $X$ admits a non-zero reflexive symmetric differential.
\end{proposition}


\end{document}